\documentclass[a4paper,12pt,reqno]{amsart}
\usepackage{amsmath,amssymb,amsthm}
\usepackage{mathrsfs}
\usepackage{enumerate}
\usepackage{graphicx}
\usepackage{mathrsfs}
\numberwithin{equation}{section}
\theoremstyle{plain}
\newtheorem{thm}{Theorem}[section]

\usepackage{mathrsfs}

\theoremstyle{Remark}
\newtheorem{Remark}[thm]{Remark}
\theoremstyle{definition}
\newtheorem{defn}[thm]{Definition}
\theoremstyle{example}

\begin{document}
\title{Continuous Wavelet Transform on Local Fields}
\author[Ashish Pathak ]{Ashish Pathak \\
 Department of Mathematics \& Statistics\\
 Dr. Harisingh Gour Central University \\
    Sagar-470003, India.}
\date{}
\subjclass[2000]{}
\thanks{$^{*}$E-mail: pathak\_maths@yahoo.com}
\begin{abstract}
  The main objective of this paper is to define the mother wavelet on local fields and study the continuous wavelet transform (CWT) and some of their basic properties. its inversion formula,  the Parseval relation  and associated convolution are also studied.
\end{abstract}
\maketitle
\section{Introduction}
 A local field means an algebraic field and a topological space with the topological
properties of locally compact, non-discrete, complete and totally disconnected, denoted
by $ \mathbb{K} $ \cite{jia}. The additive and multiplicative groups of $ \mathbb{K} $ are denoted by $ \mathbb{K}^{+} $ and $\mathbb{K}*$, respectively. We may choose a Haar measure $ dx $ for $\mathbb{K}^{+}$. If $\alpha \neq 0 (\alpha \in \mathbb{K} )$, then $d(\alpha x)$ is also a Haar measure. Let $ d(\alpha x) = |\alpha|dx $  and call  $ |\alpha| $  the absolute value or valuation of $\alpha$. Let $
|0| = 0 $. The absolute value has the following properties:\\
 (i) $|x|\geq 0 $ and $|x|=0$ if and only if $x=0$;\\
 (ii) $|xy|= |x||y|;$\\
 (iii) $|x+y| \leq max(|x|,|y|)$.\\
 The last one of these properties is called the ultrametric inequality.
 The set $  \mathfrak{D} = \{x \in  \mathbb{K} : |x| \leq  1\} $ is called the ring of integers in $ \mathbb{K} $. It is the unique
maximal compact subring of $ \mathbb{K} $. Define $ \mathfrak{P} = {x \in \mathbb{K} : |x| < 1} $. The set $\mathfrak{P}$ is called the
prime ideal in  $ \mathbb{K} $. The prime ideal in $\mathbb{K} $ is the unique maximal ideal in $\mathfrak{D}$. It is principal
and prime.  \\
 If $\mathbb{K}$ is a local field, then there is a nontrivial, unitary, continuous character $\chi$ on $\mathbb{K}^+$ and
 $ \mathbb{K}^+$ is self dual.\\
 $\chi$ is fixed character on $ \mathbb{K}^+$ that is trivial on $ \mathfrak{D} $ but is nontrivial on $ \mathfrak{P}^{-1}$ . It follows that $\chi$ is constant on cosets of $ \mathfrak{D} $ and that if $y \in \mathfrak{P}^{k} $
, then $\chi_y(\chi_y(x))=\chi(xy)$ is constant on cosets of  $ \mathfrak{P}^{-k}$
\begin{defn}
The fourier transform of $ f \in  L^1(\mathbb{K}) $ is denoted by $\hat{f}(\xi)$ and  define by the \cite{tai}
\begin{eqnarray}
\hat{f}(\xi)= \int_{\mathbb{K}} f(x) \overline{\chi_\xi(x)} dx =  \int_{\mathbb{K}} f(x) \chi(-\xi x) dx, \,\, \xi \in \mathbb{K}
\end{eqnarray}
and the inverse Fourier transform by
\begin{eqnarray}
f(x)= \int_{\mathbb{K}} \hat{f}(\xi) \chi_x(\xi) dx  \,\, x  \in  \mathbb{K}
\end{eqnarray}
Some important properties of the  Fourier transform  can prove easily :\\
(i) $||\hat{f}||_{L^{\infty}(\mathbb{K})} \leq ||\hat{f}||_{L^{1}(\mathbb{K})} $.\\
(ii) If $f \in L^1(\mathbb{K})$, then $\hat{f}$ is uniformly continuous.\\
(iii)\textbf{ Parseval formula}:If $ f \in L^1(\mathbb{K}) \cap L^2(\mathbb{K})$, then $||\hat{f}||_{L^2(\mathbb{K})}= ||f||_{L^2(\mathbb{K})} $ \\
(iv) If the convolution of $ f $ and $ g $  is defined as
\begin{eqnarray}
  (f * g)(t) = \int_{\mathbb{K}} f(x) g(t-x) dx
\end{eqnarray}
then
\begin{eqnarray}
F((f * g))= F(f).F(g)
\end{eqnarray}
\end{defn}
 The article is divided in four sections. In section 2. we proposed the definition of mother wavelet and define the continuous wavelet transform (CWT). In section 3. discus the some basic properties of CWT. In section 4. we prove the Plancherel , inversion formula and define the convolution associated with CWT.
\section{Continuous Wavelet transform on  local fields}
Similar to $L^2(\mathbb{R})$ \cite{dau,chui ,Pathak}, we  define the wavelet on local fields  and define the continuous wavelet transform.
\begin{defn}  \textbf{Admissible  wavelet on local fields} \\
  The function $\psi(x) \in L^2({\mathbb{K}})$ is said to be an admissible wavelet on local fields if $\psi(x)$ satisfies   the following admissibility condition:
\begin{eqnarray}
\label{eq  1}
c_{\psi}= \int_{\mathbb{K}} \frac{|\hat{\psi}(\xi)|^2}{|\xi|} d\xi  <  \infty
\end{eqnarray}
where $\hat{\psi}$ is the Fourier transform of $\psi$.\\
\begin{Remark}
If $|\hat{\psi}(\xi)|$ is continuous near $ \xi=0 $, then the existence of integral (\ref{eq  1})
guarantees that , $\hat{\psi}(0)=0 $.  Since the Fourier transform of  mother wavelet
$\psi \in  L^1(\mathbb{K})\cap L^2(\mathbb{K})$  is bounded and uniformly continuous, we have
\begin{eqnarray*}
\label{eq  2}
0=\hat{\psi}(0) & =& \int_\mathbb{K} \psi(x) \chi_0(x) dx \\ &=& \int_\mathbb{K} \psi(x) dx
\end{eqnarray*}
This means that the integral of  mother wavelet is zero:
\end{Remark}
\end{defn}
\begin{thm}
If $\psi$ is a mother wavelet and $\phi \in L^1(\mathbb{K})$, then the convolution function  $ \psi * \phi $ is a mother wavelet.
 \end{thm}
 \begin{proof} Since
 \begin{eqnarray*}
 \int_{\mathbb{K}} |(\psi * \phi)(x)|^2 dx &=& \int_{\mathbb{K}}  \Big\vert \int_{\mathbb{K}} \psi (x-y) \phi(y)  dy \Big\vert^2 dx  \\
 & \leq &  \int_{\mathbb{K}}  \left(  \int_{\mathbb{K}} |\psi (x-y)| \vert \phi(y)|^{\frac{1}{2}} \vert  |\phi(y)|^{\frac{1}{2}}   dy |\right)^2   dx \\
  & \leq &  \int_{\mathbb{K}}  \left(  \int_{\mathbb{K}} |\psi (x-y)| |\phi(y)| dy  \int_{\mathbb{K}}   \right) dx  \\
  &=&  \int_{\mathbb{K}} |\phi(y)| dy  \int_{\mathbb{K}} \int_{\mathbb{K}} |\psi (x-y)|^2  |\phi(y)| dy dx
  \\
  &=&  \left( \int_{\mathbb{K}} |\phi(y)| dy \right)^2  \int_{\mathbb{K}} |\psi (x)|^2 dx
  \\ &=& || \phi ||^2_{L^1({\mathbb{K}})} ||\psi||^2_{L^2({\mathbb{K}})}
 \end{eqnarray*}
 Therefore  $(\psi * \phi)(x) \in  L^2({\mathbb{K}})$. Moreover
 \begin{eqnarray*}
 c_{\psi * \phi} &= &\int_{\mathbb{K}} \frac{|\hat{\psi * \phi}(\xi)|^2}{|\xi|} d\xi  \\
 &=& \int_{\mathbb{K}} \frac{|\hat{\psi} (\xi)|^2 |\hat{\phi}(\xi)|^2}{|\xi|} d\xi \\
 & \leq & ||\hat{\phi}||^2_{L^{\infty}(\mathbb{K})} \int_{\mathbb{K}} \frac{|\hat{\psi} (\xi)|^2 }{|\xi|} d\xi
 \end{eqnarray*}
 This completes the proof of the theorem
\end{proof}
\begin{defn} \textbf{Continuous wavelet transform  (CWT) on local fields} \\
For $\psi(x) \in  L^2({\mathbb{K}}) $ and $a, b \in \mathbb{K}, a \neq 0 $, we define the unitary linear operator:
\begin{eqnarray*}
U^b_a :  L^2({\mathbb{K}}) \rightarrow  L^2({\mathbb{K}})
\end{eqnarray*}
by
\begin{eqnarray}
U^b_a(\psi(x)) = \psi_{a,b}(x) = \frac{1}{|a|^{\frac{1}{2}}} \,\, \psi ( \frac{x-b}{a} )
\end{eqnarray}
$\psi$ is called  mother wavelet and $\psi_{a,b}(x)$ are called daughter wavelets, where $ a$ is a dilation parameter,
$ b $ is a translation parameter.\\
The Fourier transform of  $\psi_{a,b}(x)$ is given by
\begin{eqnarray}
\hat{\psi_{a,b}}(\xi) = |a|^{\frac{1}{2}} \,\, \hat{\psi}(a \xi) \chi_b ( \xi)
\end{eqnarray}
where $\hat{\psi}$ is the Fourier transform of $\psi$.
\end{defn} 
The CWT on local fields
\begin{eqnarray*}
K_\psi:   L^2({\mathbb{K}}) \rightarrow  L^2({\mathbb{K}} \times {\mathbb{K}} )
\end{eqnarray*}
of a function $ f \in  L^2(\mathbb{K}) $ with respect to a mother wavelet  $\psi$ is defined by
\begin{eqnarray}
f \mapsto K_\psi f(a,b) & =&  \nonumber (f,\psi_{a,b} )_{L^2(\mathbb{K})} \\
 &=& \nonumber \int_{\mathbb{K}} f(x) \overline{\psi_{a,b}(x)} dx \\
 &=& \int_{\mathbb{K}} f(x)  \frac{1}{|a|^{\frac{1}{2}}} \,\, \overline{ \psi ( \frac{x-b}{a} )} dx
\end{eqnarray}
\section {Basic Properties of CWT on local fields}
Before giving the fundamental properties of CWT, we list their basic properties.
\begin{thm}
Let $\psi$ and $ \varphi $ be to wavelets and $f,g$ are two function belong to $L^2(\mathbb{K})$, then \\
(1) \textbf{Linearity }
 \begin{eqnarray}
 K_\psi( \eta f + \vartheta g) (a,b)  =  \eta K_\psi(f) (a,b)+  \vartheta K_\psi(g) (a,b)
\end{eqnarray}
where $\eta$ and $\vartheta$ are any two scalers. \\
(ii)\textbf{Shift property}
\begin{eqnarray}
 K_\psi( f(x-\varsigma)) (a,b)  =   K_\psi(f) (a,b-\varsigma)
 \end{eqnarray}
 where $\varsigma $  is any scalers.\\
(iii) \textbf{Scaling property}
If $ \sigma \neq 0 $ any scaler.  The CWT of the scaled function $ f_\sigma(x) = \frac{1}{\sigma} f(\frac{1}{\sigma}) $ is
\begin{eqnarray}
 K_\psi( f_\sigma(x)) (a,b)  =   K_\psi(f) (\frac{a}{\sigma},\frac{b}{\sigma})
 \end{eqnarray}
 (iv) \textbf{Symmetry }
 \begin{eqnarray}
 K_\psi(f) (a,b)  =  \overline{ K_f (\psi) (\frac{1}{a},-\frac{1}{b})}
 \end{eqnarray}
 (iv) \textbf{Parity }
 \begin{eqnarray}
 K_{P(\psi)}(Pf) (a,b)  =  K_\psi (f) (a,-b)
 \end{eqnarray}
 where $P$ is a parity operator define by $Pf(x)=f(-x)$.
\end{thm}
\begin{proof}
The proof is the straight forward application of CWT
\end{proof}
\begin{thm}
Show that the continuous wavelet transform can also expressed as
\begin{eqnarray}
(K_\psi f )(a,b) = \left(  f * \frac{1}{\sqrt{|a|}} \overline{\psi} (\frac{x}{a}) \right)  (b)
\end{eqnarray}
where the $*$ is defined as
\begin{eqnarray}
  (f * g)(t) = \int_{\mathbb{K}} f(x) g(t-x) dx
\end{eqnarray}
\end{thm}
\begin{proof} From define of CWT we have
\begin{eqnarray}
(K_\psi f )(a,b) & = & \nonumber  \int_{\mathbb{K}} f(x)  \frac{1}{|a|^{\frac{1}{2}}} \,\, \overline{ \psi ( \frac{x-b}{a} )} dx \\
&=& \left(  f * \frac{1}{\sqrt{|a|}} \overline{\psi} (\frac{x}{a}) \right)  (b)
\end{eqnarray}
\end{proof}
\begin{thm}
if $f$ is homogeneous function of degree n show that
\begin{eqnarray}
(K_\psi f )(\lambda a,\lambda b) = \lambda^{n} |\lambda |^{\frac{1}{2}} (K_\psi f )(a, b)
\end{eqnarray}
where the $\lambda $ is  scaler .
\end{thm}
\begin{proof} From define of CWT we have
\begin{eqnarray}
(K_\psi f )(\lambda a,\lambda b) & = \nonumber & \int_{\mathbb{K}} f(x)  \frac{1}{|\lambda a|^{\frac{1}{2}}} \,\, \overline{ \psi ( \frac{x-\lambda b}{\lambda a} )} dx \\
&=& \nonumber \int_{\mathbb{K}} f(\lambda x)  \frac{1}{|a|^{\frac{1}{2}}} \,\, \overline{ \psi ( \frac{x- b}{ a} )} |\lambda|dx
\\
&=& \lambda^{n} |\lambda |^{\frac{1}{2}} (K_\psi f )(a, b)
\end{eqnarray}
\end{proof}
\section{Main Properties of the CWT}
This section describes important properties of the CWT, such as the Plancherel ,  inversion formula and associated convolution  fist, we establish the Plancherel theorem.
\begin{thm} (\textbf{QFT Plancherel}) Let $f,g \in L^2(\mathbb{K})$. Then we have
\begin{eqnarray}
 (K_\psi(f)(a,b), K_\psi(g)(a,b) )_{L^2(\mathbb{K}\times \mathbb{K})}= c_\psi (f,g)_{L^2(\mathbb{K})}
 \end{eqnarray}
 where $c_\psi$ is given in (2.1).
\end{thm}
\begin{proof}
By using perseval formula for Fourier we can write the wavelet transform as
\begin{eqnarray}
 K_\psi(f)(a,b) &=&  \nonumber  \int_{\mathbb{K}} f(x)  \frac{1}{|a|^{\frac{1}{2}}} \,\, \overline{ \psi ( \frac{x-b}{a} )} dx \\
 &=& \nonumber  (f, \psi_{a,b}) \\
 &=& \nonumber (\hat{f}, \hat{\psi_{a,b}})\\
 &=&  \int_{\mathbb{K}} \hat{f}(\xi)  |a|^{\frac{1}{2}} \,\, \overline{\hat{\psi}(a \xi)} \overline{ \chi_b ( \xi)}  d\xi
 \end{eqnarray}
 Similarly
  \begin{eqnarray}
\overline{K_\psi(g)(a,b)}  &=& \int_{\mathbb{K}} \overline{\hat{f}(\xi)}  |a|^{\frac{1}{2}} \,\, \hat{\psi}(a \xi) \chi_b ( \xi)  d\xi
 \end{eqnarray}
 Now, by using above (4.2) and (4.3) we get
 \begin{eqnarray}
\int_{\mathbb{K}} \int_{\mathbb{K}} K_\psi(f)(a,b) \overline{K_\psi(g)(a,b)} \frac{da db}{|a|^2}
& = &  \int_{\mathbb{K}} \int_{\mathbb{K}}  |a|  \frac{da db}{|a|^2}  \int_{\mathbb{K}}    \hat{f}(\xi)  \overline{\hat{\psi}(a \xi) \chi_b(\xi) }  d\xi \nonumber \\
&& \times  \int_{\mathbb{K}} \overline{\hat{g}(\upsilon)} \hat{\psi}(a \upsilon )\chi_b (\upsilon)  d\upsilon \nonumber \\
&=&  \int_{\mathbb{K}} \int_{\mathbb{K}} \frac{da db} {|a|}  \overline{ \int_{\mathbb{K}}  \overline{\hat{f}(\xi)}  \hat{\psi}(a \xi) \chi_b(\xi)  d\xi} \nonumber \\
&=&  \int_{\mathbb{K}} \int_{\mathbb{K}}   \overline{F(\overline{\hat{f}(\xi)} \hat{\psi}(a \xi))}(b) F(\overline{\hat{g}(\upsilon)} \hat{\psi}(a \upsilon ))(b) \frac{da db} {|a|} \nonumber \\
&=&  \int_{\mathbb{K}} \int_{\mathbb{K}}  \hat{f}(\xi) \overline{\hat{\psi}(a \xi)} \overline{\hat{g}(\xi)} \hat{\psi}(a\xi) \frac{d \xi da}{|a|}
 \nonumber \\
 &=&  \int_{\mathbb{K}}   \hat{f}(\xi)  \overline{\hat{g}(\xi)} \left(   \int_{\mathbb{K}} \overline{\hat{\psi}(a \xi)}  \hat{\psi}(a\xi) \frac{da}{|a|} \right)  d\xi
 \nonumber \\
 &=&  \int_{\mathbb{K}} \hat{f}(\xi)  \overline{\hat{g}(\xi)} \left(  \int_{\mathbb{K}} \frac{|\hat{\psi}(a \xi)|^2}{|a|} da \right)  d\xi
\end{eqnarray}
 \begin{eqnarray}
  &=&  \int_{\mathbb{K}} \hat{f}(\xi)  \overline{\hat{g}(\xi)} \left(  \int_{\mathbb{K}} \frac{|\hat{\psi}(\omega)|^2}{|\omega|} dz \right)  d\xi
 \nonumber \\
& = & c_\psi (\hat{f},\hat{g})_{L^2(\mathbb{K})}  \nonumber \\
& = & c_\psi (f,g)_{L^2(\mathbb{K})}
\end{eqnarray}
\end{proof}
\begin{thm} (\textbf{Inversion Formula}) Let $f \in L^2(\mathbb{K})$. Then we have
\begin{eqnarray}
 f(x)= \frac{1}{c_\psi} \int_{\mathbb{K}} \int_{\mathbb{K}} K_\psi(f)(a,b) \psi_{a,b}(x) \frac{da db}{|a|^2}
 \end{eqnarray}
 where $c_\psi$ is given in (2.1).
\end{thm}
\begin{proof}
Let $h(x) \in L^2(\mathbb{K})$ be any function, then by using above theorem, we have
\begin{eqnarray}
c_\psi (f,g)_{L^2(\mathbb{K})} \nonumber  & = & \int_{\mathbb{K}} \int_{\mathbb{K}} ( K_\psi f)(a,b) \overline{K_\psi(h)(a,b)} \frac{da db}{|a|^2} \\ \nonumber  & = & \int_{\mathbb{K}} \int_{\mathbb{K}} (K_\psi f)(a,b) \overline{\int_{\mathbb{K}} h(x) \overline{\psi_{a,b}(x)} dx} \frac{da db}{|a|^2}  \\ \nonumber  & = & \int_{\mathbb{K}}  \int_{\mathbb{K}} \int_{\mathbb{K}} (K_\psi f)(a,b)  \psi_{a,b}(x) \overline{ h(x)} \frac{da db dx}{|a|^2}  \\
\nonumber &=& \left(   \int_{\mathbb{K}} \int_{\mathbb{K}} (K_\psi f)(a,b)  \psi_{a,b}(x) \frac{da db}{|a|^2} ,h(x) \right)
\end{eqnarray}
Hence the result follows.
\end{proof}
If $ f=h $
\begin{eqnarray}
|| f||^2_{L^2(\mathbb{K})} =  \int_{\mathbb{K}} \int_{\mathbb{K}} |(K_\psi f)(a,b)|^2 \frac{da db}{|a|^2}
\end{eqnarray}
Moreover the wavelet transform is isometry from $L^2(\mathbb{K})$ to  $L^2(\mathbb{K} \times \mathbb{K})$ \\
\subsection{ Associated convolution  for CWT on local fields }
 Using Pathak and Pathak  techniques \cite{Pathak}, we define the basic function $D(x,y,z)$ , translation ${\tau_x}$ and associated  convolution $\# $ operators for CWT.\\
 The basic function \,$ D(x,y,z)$ for (2.4) is define as
\begin{eqnarray}
\label{2.1}
K_\phi[D(x,y,z)](a,b) &=&\int_{\mathbb{K}} D(x,y,z)\overline{{\phi}_{a,b}(t)}dt\nonumber \\
 &=& \overline{\psi_{a,b}(z)}\,\ \overline{\chi_{a,b}(y)},
\end{eqnarray}
where $ \psi , \phi $ and $ \chi $ are three wavelets satisfying certain conditions (2.1).\\
Now, by  using(4.6) we get,
\begin{equation}
\label{2.2}
D(x,y,z)= C^{-1}_\phi \int_{\mathbb{K}}\int_{\mathbb{K}}\overline{\psi_{a,b}(z)}\,\ \overline{\chi_{a,b}(y)}\ {\phi_{a,b}(x)} \, |a|^{-2}da db.
\end{equation}
The translation  $\tau_x $ is defined as \cite{Pathak}
\begin{eqnarray*}
{(\tau_x h)(y)} & = & h^*(x,y)=\int_{\mathbb{K}} D(x,y,z)h(z)dz
\\ &=& C^{-1}_\phi\int_{\mathbb{K}}\int_{\mathbb{K}}\int_{\mathbb{K}}\overline{\psi_{a,b} (z)}\,\, \overline{\chi_{a,b}(y)}\,\, \phi_{a,b}(x)\, h(z)|a|^{-2}dadbdz.
\end{eqnarray*}
The associated  convolution  is defined  as
\begin{eqnarray}
\label{2.3}
(h\# g)(x)\nonumber &=& \int_{\mathbb{K}}h^*(x,y) g(y)dy \\ \nonumber &=& \int_{\mathbb{K}}\int_{\mathbb{K}}D(x,y,z)\,h(z)\,g(y)dy dz\\ &=& C^{-1}_\phi \int_{\mathbb{K}}\int_{\mathbb{K}}\int_{\mathbb{K}}\int_{\mathbb{K}} \overline{{\psi}_{a,b}(z)}\,\,  \overline{{\chi}_{a,b}(y)} \,\,\phi_{a,b}(x) h(z)g(y)\left|a\right|^{-2}dadbdzdy,
\end{eqnarray}
by using the inversion formula we can write the above equation as
\begin{eqnarray}
(h\# g)(x)   \nonumber &=& C^{-1}_\phi \int_{\mathbb{K}}\int_{\mathbb{K}} (K_\psi h)(a,b) (K_\chi g)(a,b) \phi_{a,b}(x) |a|^{-2}dadb
\\
&=&  \nonumber  K^{-1}_\phi \left[ (K_\psi h)(a,b)  (K_\chi g)(a,b)  \right] (x)
\end{eqnarray}
so that
\begin{eqnarray}
K_\phi[h\# g](b,a)   =   (K_\psi h)(a,b)  (K_\chi g)(a,b) (x)
\end{eqnarray}
\section*{Acknowledgment}
The work is supported by U.G.C start-up grant No:F.30-12(2014)/(BSR).

\end{document}